\documentclass[12pt]{article}

\usepackage{amsmath, amssymb, amsthm}
\usepackage{mathrsfs}
\usepackage{enumitem}
\usepackage{hyperref}
\usepackage{geometry}
\usepackage{bm}
\usepackage{algorithm}
\usepackage{algpseudocode}
\usepackage{tikz}
\usetikzlibrary{arrows.meta, positioning}
\usepackage{stmaryrd}
\usepackage{booktabs}
\usepackage{graphicx}
\usepackage{color}
\usepackage{tikz}

\newcommand{\keywords}[1]{%
  \par\noindent\textbf{Keywords.} #1}

\newcommand{\msc}[1]{%
  \par\noindent\textbf{MSC (2020).} #1}

\usepackage[numbers]{natbib}

\newtheorem{theorem}{Theorem}[section]
\newtheorem{lemma}[theorem]{Lemma}

\newtheorem{definition}[theorem]{Definition}
\newtheorem{remark}[theorem]{Remark}
\newtheorem{example}[theorem]{Example}

\title{Exact Formulas for Coprime Representations of Even Integers Avoiding a Prime}

\author{Andrés M. Salazar \thanks{Universidad Javeriana Cali - Colombia, andresmsalazar@javerianacali.edu.co}}
\date{\today}

\begin{document}

\maketitle

\begin{abstract}
Fix a prime $p \ge 5$ and define
\[
g(2n,p)
= \#\bigl\{(h,k)\in\mathbb{Z}_{>0}^2 :
h+k=2n,\; h\le k,\;
\gcd(h,6p)=\gcd(k,6p)=1\bigr\}.
\]
We derive explicit closed-form expressions for $g(2n,p)$ in terms
of the canonical remainder operator $\delta_k(x)=x-k\lfloor x/k\rfloor$,
elementary step functions, and the minimal solutions of the
congruences $6x \equiv -1 \pmod{p}$ and $6x \equiv -5 \pmod{p}$.
The key ingredient is an explicit formula for the minimal solution
of $\delta_k(a_0 x)=b_0$ obtained via the Euclidean algorithm,
which determines the excluded residue classes directly.
The resulting formulas reveal that $g(2n,p)$ is piecewise affine
along arithmetic progressions of $n$, governed by residue classes
modulo $3$ and $p$. For fixed $p$, after precomputing two residue
parameters in $O(\log p)$ time, each evaluation of $g(2n,p)$
requires only $O(1)$ operations, compared to $O(n)$ for direct
enumeration. The formulas are validated computationally for all
$2n \le 10^5$ and primes $p \in \{5,7,11,13,17,19,23\}$, with
perfect agreement with brute-force enumeration.

\keywords{additive number theory; coprime representations;
exact counting formulas; remainder operator;
combinatorial number theory}

\msc{11A41, 11B75, 11N25, 11P99}
\end{abstract}

\section{Introduction}
\label{sec:intro}

The enumeration of representations of an even integer as a sum of
two positive integers subject to coprimality conditions is a
classical problem in additive number theory
\cite{HardyWright2008, Nathanson2000}.
While asymptotic estimates for related counting functions follow
from sieve methods and Möbius inversion
\cite{halberstam-richert, iwaniec-kowalski},
obtaining fully explicit exact formulas is a more delicate task,
as it requires precise control of residue constraints and ordering
conditions.

In this paper we study the counting function
\begin{equation}\label{eq:main}
g(2n,p)
= \#\bigl\{(h,k)\in\mathbb{Z}_{>0}^2 :
h+k=2n,\; h\le k,\;
\gcd(h,6p)=\gcd(k,6p)=1\bigr\},
\end{equation}
where $p \ge 5$ is a fixed prime. This function counts representations
of $2n$ as a sum $h+k$ of two positive integers with $h \le k$,
both avoiding the primes $2$, $3$, and $p$.

The function $g(2n,p)$ can be evaluated via inclusion--exclusion
over the prime divisors of $6p$, yielding a piecewise expression
periodic modulo $6p$, as follows from standard divisor-sum arguments.
However, such formulas typically involve alternating sums over divisor
sets and do not directly reveal the affine structure of the counting
function. Our main result, Theorem~\ref{thm:main}, provides an
alternative approach: explicit closed formulas expressed in terms of
the canonical remainder operator
\[
\delta_k(x) := x - k\left\lfloor\frac{x}{k}\right\rfloor,
\]
elementary step functions, and the minimal solutions $a(p)$ and $b(p)$
of the congruences
\[
6x \equiv -1 \pmod{p},
\qquad
6x \equiv -5 \pmod{p}.
\]
These residue parameters determine the excluded congruence classes
modulo $p$ in a direct and explicit manner.

A key technical component is the explicit resolution of affine residue
equations of the form $\delta_k(a_0 x + b) = c$, which reduce to
equations of the form $\delta_k(a_0 x) = b_0$ via additive compatibility.
This reduction is applied in Section~\ref{sec:main} to determine the
precise residue exclusions governing admissible coprime decompositions.

For fixed $p$, after precomputing $a(p)$ and $b(p)$ in $O(\log p)$
operations using the Euclidean algorithm, each evaluation of $g(2n,p)$
requires only $O(1)$ arithmetic operations. This represents a substantial
improvement over direct enumeration, which requires $O(n)$ operations.

The resulting formulas reveal that $g(2n,p)$ is piecewise affine along
arithmetic progressions of $n$, with structure determined by residue
classes modulo $3$ and $p$. This affine structure is not transparent
from inclusion--exclusion formulas or asymptotic estimates, and emerges
naturally from the residue-based formulation developed here.

Theorem~\ref{thm:main} is validated computationally against a brute-force
oracle for all even integers $2n \le 10^5$ and primes
$p \in \{5,7,11,13,17,19,23\}$, with perfect agreement observed in
all cases.

The paper is organized as follows.
Section~\ref{sec:preliminaries} fixes notation and collects standard
properties of the remainder operator $\delta_k$.
Section~\ref{sec:euclidean} derives an explicit formula for the minimal
solution of $\delta_k(a_0 x)=b_0$ via the Euclidean algorithm.
Section~\ref{sec:main} applies these tools to obtain
Theorem~\ref{thm:main}.
Section~\ref{sec:validation} presents the computational validation.

\section{Preliminaries}\label{sec:preliminaries}

For integers $a,b \in \mathbb{Z}$, we denote by $\gcd(a,b)$ their greatest 
common divisor, and say that $a$ and $b$ are \emph{coprime} if $\gcd(a,b)=1$. 
We write $a \mid b$ when $a$ divides $b$, and $a \nmid b$ otherwise. For 
integer intervals, we use the notation
\[
[a,b]_{\mathbb{Z}} := \{\, n \in \mathbb{Z} : a \le n \le b \,\}.
\]

\subsection{Remainder operator}

For $k \ge 2$, define the map
\[
\delta_k : \mathbb{Z} \to [0,k-1]_{\mathbb{Z}},
\qquad
\delta_k(a) := a - k \left\lfloor \frac{a}{k} \right\rfloor .
\]

Thus $\delta_k(a)$ is the canonical representative of the residue class of $a$ 
modulo $k$, and satisfies
\[
\delta_k(a) \in \{0,1,\dots,k-1\},
\qquad
\delta_k(a)=0 \iff k \mid a.
\]

The operator $\delta_k$ coincides with the canonical projection 
$\mathbb{Z} \to \mathbb{Z}/k\mathbb{Z}$ followed by the choice of the standard 
representative. In particular, for all $a,b,r,n \in \mathbb{Z}$, the following 
identities hold:
\begin{align}
\delta_k(a+kn) &= \delta_k(a), \label{eq:period}\\
\delta_k(a+b)  &= \delta_k\!\left(\delta_k(a)+\delta_k(b)\right),
                \label{eq:addcomp}\\
\delta_k(ra)   &= \delta_k\!\left(r\,\delta_k(a)\right).
                \label{eq:multcomp}
\end{align}

These identities express compatibility with addition and multiplication 
after projection to canonical representatives; they follow directly from 
the definition of $\delta_k$ and standard properties of the floor function 
(see, e.g.,~\cite{HardyWright2008}). For example,
\[
\delta_k(ab+c) = \delta_k\!\left(a\,\delta_k(b)+\delta_k(c)\right).
\]

We also record the idempotence property.

\begin{lemma}\label{lem:idempotent}
For all $a \in \mathbb{Z}$ and $k \ge 2$,
\[
\delta_k(\delta_k(a))=\delta_k(a).
\]
\end{lemma}

\begin{proof}
Since $\delta_k(a) \in [0,k-1]_{\mathbb{Z}}$, we have
\[
\left\lfloor \frac{\delta_k(a)}{k} \right\rfloor = 0,
\]
and the definition immediately gives the result.
\end{proof}

The following auxiliary lemmas will be used throughout the paper.

\begin{lemma}\label{lema:cp2}
Let $k \ge 2$ and $a \in \mathbb{Z}$ with $\delta_k(a) \neq 0$. Then
\[
\delta_k(-a) = k - \delta_k(a).
\]
\end{lemma}

\begin{proof}
By~\eqref{eq:addcomp},
\[
\delta_k\!\left(\delta_k(a) + \delta_k(-a)\right)
= \delta_k(a + (-a))
= \delta_k(0)
= 0,
\]
so $k \mid \bigl(\delta_k(a) + \delta_k(-a)\bigr)$.
Since both $\delta_k(a)$ and $\delta_k(-a)$ lie in $\{1,\dots,k-1\}$,
their sum belongs to $\{2,\dots,2k-2\}$, and the only multiple of $k$
in this interval is $k$ itself. Therefore $\delta_k(-a) = k - \delta_k(a)$.
\end{proof}

\begin{lemma}\label{lema:despeje}
Let $k \ge 2$ and $a,b \in \mathbb{Z}$. Then
\[
\delta_k(a) = \delta_k(b)
\quad\Longleftrightarrow\quad
\delta_k(a-b) = 0.
\]
\end{lemma}

\begin{proof}
If $\delta_k(a)=\delta_k(b)$, then by~\eqref{eq:addcomp},
\[
\delta_k(a-b)
=
\delta_k\!\left(\delta_k(a)-\delta_k(b)\right)
=
\delta_k(0)
=
0.
\]
Conversely, if $\delta_k(a-b)=0$, then
\[
\delta_k\!\left(\delta_k(a)-\delta_k(b)\right)=0.
\]
Since
\[
\delta_k(a)-\delta_k(b)
\in [-(k-1),k-1]_{\mathbb{Z}},
\]
and the only multiple of $k$ in this interval is $0$, it follows that
\[
\delta_k(a)=\delta_k(b).
\]
\end{proof}

\begin{lemma}\label{lema_pote}
Let $k \ge 2$ and $r \in \mathbb{Z}$ with $\gcd(r,k)=1$. Then for any 
$a \in \mathbb{Z}$,
\[
\delta_k(ar) = \delta_k(r)
\quad\Longleftrightarrow\quad
\delta_k(a) = 1.
\]
\end{lemma}

\begin{proof}
If $\delta_k(a)=1$, then~\eqref{eq:multcomp} gives
\[
\delta_k(ar)
=
\delta_k(r).
\]
Conversely, if $\delta_k(ar)=\delta_k(r)$, then by 
Lemma~\ref{lema:despeje},
\[
\delta_k(r(1-a))=0.
\]
Since $\gcd(r,k)=1$, multiplication by $r$ is invertible modulo $k$, and 
therefore $\delta_k(1-a)=0$, which implies $\delta_k(a)=1$.
\end{proof}

\begin{lemma}\label{lem:conmuta1}
Let $k,r \ge 2$ and $a \in \mathbb{Z}$. If $r \mid k$, then
\[
\delta_r\bigl(\delta_k(a)\bigr) = \delta_r(a).
\]
\end{lemma}

\begin{proof}
From the definition,
\[
\delta_k(a)
=
a - k \left\lfloor \frac{a}{k} \right\rfloor.
\]
Since $r \mid k$, the second term is divisible by $r$, and applying 
$\delta_r$ gives
\[
\delta_r(\delta_k(a))
=
\delta_r(a).
\]
\end{proof}

\subsection{Minimal solutions of linear remainder equations}
\label{subsec:minimal}

We study integer solutions of equations of the form
\begin{equation}\label{eq:general}
\delta_k(ax+b) = c,
\end{equation}
where $a,b \in \mathbb{Z}$, $c \in \{0,1,\dots,k-1\}$, and $k \ge 2$.

By definition of $\delta_k$, equation~\eqref{eq:general} holds if and only if
\[
k \mid (ax+b-c).
\]
This formulation allows us to study remainder equations entirely within $\mathbb{Z}$,
without introducing residue classes as separate algebraic objects.

\begin{lemma}[Existence and periodicity]\label{lema:criterio_funcional}
Let $k \ge 2$, $a,b \in \mathbb{Z}$, $c \in \{0,\dots,k-1\}$, and let
\[
r := \gcd(a,k).
\]
Then equation~\eqref{eq:general} admits an integer solution if and only if
\[
r \mid (c-b).
\]
In that case, the set of all solutions is an arithmetic progression
\[
x = x_0 + t\,\frac{k}{r}, \qquad t \in \mathbb{Z},
\]
for some particular solution $x_0$.
\end{lemma}

\begin{proof}
Equation~\eqref{eq:general} is equivalent to
\[
ax \equiv c-b \pmod{k},
\]
which in turn is equivalent to the linear Diophantine equation
$ax - kt = c-b$ for some integer $t$.
It is a standard result that this equation has an integer solution if and only
if $\gcd(a,k) \mid (c-b)$ (see, for example, \cite{HardyWright2008}).
When solutions exist, if $x_0$ is one solution, then
\[
a\!\left(x_0 + \frac{k}{r}\right) = ax_0 + \frac{ak}{r} \equiv ax_0 \pmod{k},
\]
since $ak/r$ is divisible by $k$. Thus all solutions form an arithmetic
progression with common difference $k/r$.
\end{proof}

\begin{definition}[Minimal solution]\label{def:minima}
Assume that equation~\eqref{eq:general} admits at least one integer solution.
The \emph{minimal solution} is defined as
\[
x_0 :=
\min\bigl\{x \in \{0,1,\dots,k-1\} : \delta_k(ax+b)=c\bigr\}.
\]
\end{definition}

\begin{lemma}[Solutions in $\{0,\dots,k-1\}$]\label{lem:minimal_unique}
Let $r = \gcd(a,k)$. If equation~\eqref{eq:general} admits an integer
solution, then it admits exactly $r$ solutions in $\{0,1,\dots,k-1\}$,
forming an arithmetic progression with common difference $k/r$.
In particular, the minimal solution exists and is unique.
\end{lemma}

\begin{proof}
By Lemma~\ref{lema:criterio_funcional}, all integer solutions have the form
\[
x = x_0 + t\,\frac{k}{r}, \qquad t \in \mathbb{Z}.
\]
The representatives in $\{0,\dots,k-1\}$ correspond to those values of $t$
for which $0 \le x_0 + t\,k/r \le k-1$.
Since consecutive representatives differ by $k/r$, and
$r \cdot (k/r) = k$, exactly $r$ distinct values lie in $\{0,\dots,k-1\}$.
These are
\[
\delta_k(x_0),\;
\delta_k\!\left(x_0 + \tfrac{k}{r}\right),\;
\dots,\;
\delta_k\!\left(x_0 + (r-1)\tfrac{k}{r}\right),
\]
and they are distinct because any two differ by a nonzero multiple of $k/r$
less than $k$.
Hence the minimal solution, being the smallest among these $r$ values,
exists and is unique.
\end{proof}

\begin{lemma}[Uniqueness under coprimality]\label{lema:unicidad_ep}
Let $k \ge 2$, $a,b \in \mathbb{Z}$, $c \in \{0,\dots,k-1\}$, and suppose
$\gcd(a,k)=1$. Then equation~\eqref{eq:general} has exactly one solution
in $\{0,\dots,k-1\}$, which coincides with the minimal solution.
\end{lemma}

\begin{proof}
Since $\gcd(a,k)=1$, Lemma~\ref{lem:minimal_unique} gives exactly one
solution in $\{0,\dots,k-1\}$. Alternatively, since multiplication by $a$
is invertible modulo $k$, the congruence $ax \equiv c-b \pmod{k}$ has a
unique solution modulo $k$, whose canonical representative is the minimal
solution.
\end{proof}

\section{The Euclidean Algorithm and Minimal Solutions}
\label{sec:euclidean}

Lemma~\ref{lema:unicidad_ep} guarantees that, under the coprimality
condition $\gcd(a_0,k)=1$, the equation $\delta_k(a_0 x)=b_0$ has
a unique solution in $\{0,\dots,k-1\}$. In this section we derive
an explicit formula for that solution via the Euclidean algorithm,
expressed using the remainder operator $\delta_k$.

\begin{lemma}\label{lem_eu}
Let $k \ge 2$ and $a_0 \in \mathbb{Z}$ with $2 \le a_0 < k$. Define
\[
a^{(0)} = k,
\quad
a^{(1)} = a_0,
\quad
a^{(n)} =
\delta_{a^{(n-1)}}\!\bigl(a^{(n-2)}\bigr),
\qquad n \ge 2.
\]
Then:
\begin{enumerate}
\item The sequence satisfies
\[
0 \le a^{(n)} < a^{(n-1)}
\quad \text{whenever } a^{(n-1)} \ne 0.
\]
\item There exists a unique index $r \ge 1$ such that
\[
a^{(r)} = \gcd(a_0,k),
\quad
a^{(r+1)} = 0.
\]
\item The algorithm terminates in $O(\log\min\{k,a_0\})$ steps.
\end{enumerate}
\end{lemma}

\begin{proof}
\textit{Item 1.} By definition of $\delta$, for all $n \ge 2$,
\[
0 \le \delta_{a^{(n-1)}}(a^{(n-2)}) < a^{(n-1)}.
\]
Thus the sequence is strictly decreasing until reaching $0$, and must
terminate. Uniqueness of $r$ follows from this strict decrease.

\textit{Item 2.}
Let $a^{(r)}$ be the last nonzero term, so $a^{(r+1)} = 
\delta_{a^{(r)}}(a^{(r-1)}) = 0$, which gives $a^{(r)} \mid a^{(r-1)}$.

We claim $a^{(r)}$ divides all terms $a^{(0)}, \dots, a^{(r-1)}$.
We proceed by downward induction. The base cases are
$a^{(r)} \mid a^{(r)}$ and $a^{(r)} \mid a^{(r-1)}$.
Suppose $a^{(r)} \mid a^{(n)}$ and $a^{(r)} \mid a^{(n-1)}$ for some
$r \ge n \ge 2$. Since $a^{(n)} = \delta_{a^{(n-1)}}(a^{(n-2)})$ and
$a^{(r)} \mid a^{(n-1)}$, Lemma~\ref{lem:conmuta1} gives
\[
\delta_{a^{(r)}}\!\left(a^{(n-2)}\right)
= \delta_{a^{(r)}}\!\left(\delta_{a^{(n-1)}}(a^{(n-2)})\right)
= \delta_{a^{(r)}}\!\left(a^{(n)}\right) = 0,
\]
so $a^{(r)} \mid a^{(n-2)}$. By induction, $a^{(r)}$ divides all terms,
hence $a^{(r)} \mid k$ and $a^{(r)} \mid a_0$,
so $a^{(r)} \le \gcd(a_0,k)$.

Conversely, let $d = \gcd(a_0,k)$. Then $d \mid a^{(0)} = k$ and
$d \mid a^{(1)} = a_0$. Suppose $d \mid a^{(n-2)}$ and $d \mid a^{(n-1)}$
for some $n \ge 2$. Since $d \mid a^{(n-1)}$, Lemma~\ref{lem:conmuta1}
gives
\[
\delta_{d}\!\left(a^{(n)}\right)
= \delta_{d}\!\left(\delta_{a^{(n-1)}}(a^{(n-2)})\right)
= \delta_{d}\!\left(a^{(n-2)}\right) = 0,
\]
so $d \mid a^{(n)}$ for all $n$ by induction. In particular
$d \mid a^{(r)}$, so $d \le a^{(r)}$.

Combining both inequalities, $a^{(r)} = \gcd(a_0,k)$.

\textit{Item 3.}
The complexity bound follows from the standard analysis of the Euclidean
algorithm \cite{BrentZimmermann2010}.
\end{proof}

\subsection{Explicit formula for the minimal solution}

The following result provides an explicit closed formula for the minimal solution,
obtained by expressing the extended Euclidean algorithm entirely in terms of the
remainder operator $\delta_k$.

\begin{theorem}\label{thm:min_sol_afrc}
Let
\[
k \ge 2,
\quad
\gcd(a_0,k)=1,
\quad
b_0 \in \{0,\dots,k-1\}.
\]
Define sequences
\[
a^{(0)}=k,\quad a^{(1)}=a_0,\quad
a^{(n)}=\delta_{a^{(n-1)}}(a^{(n-2)}),
\]
and
\[
b^{(1)}=b_0,\quad
b^{(n)}=a^{(n-1)}-\delta_{a^{(n-1)}}(b^{(n-1)}).
\]
Let $r$ be the unique index with $a^{(r)}=1$ (cf.\ Lemma~\ref{lem_eu}).
Then the unique solution in $\{0,\dots,k-1\}$ of
\[
\delta_k(a_0 x)=b_0
\]
is
\[
x =
\delta_k\!\left(
k\sum_{j=0}^{r-1}
\frac{b^{(j+1)}}{a^{(j)}a^{(j+1)}}
\right).
\]
\end{theorem}

\begin{proof}
If $b_0 = 0$, then $x = 0$ is the unique solution in $\{0,\dots,k-1\}$,
and the formula gives $\delta_k(0) = 0$, consistently.
Henceforth assume $b_0 \ge 1$.

Since $\gcd(a_0,k)=1$, by Lemma~\ref{lema:unicidad_ep} the equation
admits a unique solution in $\{0,\dots,k-1\}$.

\medskip
\textit{Step 1: Construction of the auxiliary system.}

Set $x_1 := x$ and write the equation as
\[
\delta_{a^{(0)}}(a^{(1)} x_1) = b^{(1)}.
\]
By definition of $\delta_{a^{(0)}}$, there exists an integer $x_2$ such
that
\[
a^{(1)} x_1 = a^{(0)} x_2 + b^{(1)}.
\tag{$R_2$}
\]

We claim that for each $n = 2, \dots, r$ there exists an integer $x_n$
satisfying
\[
a^{(n-1)} x_{n-1} = a^{(n-2)} x_n + b^{(n-1)}.
\tag{$R_n$}
\]

We proceed by induction on $n$.

\textit{Base case} $n = 2$: this is $(R_2)$ already established.

\textit{Inductive step}: assume $(R_n)$ holds for some $2 \le n < r$,
so
\[
a^{(n-1)} x_{n-1} = a^{(n-2)} x_n + b^{(n-1)}.
\]
Since $a^{(n-2)} x_n + b^{(n-1)}$ is divisible by $a^{(n-1)}$,
applying $\delta_{a^{(n-1)}}$ to both sides and using additive
compatibility gives
\[
0 = \delta_{a^{(n-1)}}\!\left(a^{(n-2)} x_n + b^{(n-1)}\right)
  = \delta_{a^{(n-1)}}\!\left(\delta_{a^{(n-1)}}(a^{(n-2)})\, x_n
    + \delta_{a^{(n-1)}}(b^{(n-1)})\right).
\]
Setting $a^{(n)} := \delta_{a^{(n-1)}}(a^{(n-2)})$, this becomes
\[
\delta_{a^{(n-1)}}(a^{(n)} x_n) = a^{(n-1)} - \delta_{a^{(n-1)}}(b^{(n-1)})
= b^{(n)}.
\]
By definition of $\delta_{a^{(n-1)}}$, there exists an integer $x_{n+1}$
such that
\[
a^{(n)} x_n = a^{(n-1)} x_{n+1} + b^{(n)},
\]
which is $(R_{n+1})$. This completes the induction.

\medskip
\textit{Step 2: Base case of back-substitution.}

At step $n = r$, relation $(R_r)$ reads
\[
a^{(r-1)} x_{r-1} = a^{(r-2)} x_r + b^{(r-1)}.
\]
Since $a^{(r)} = \delta_{a^{(r-1)}}(a^{(r-2)}) = 1$, the equation
$\delta_{a^{(r-1)}}(a^{(r)} x_{r-1}) = b^{(r)}$ reduces to
$\delta_{a^{(r-1)}}(x_{r-1}) = b^{(r)}$,
so we may set
\[
x_r := b^{(r)} = a^{(r-1)} - \delta_{a^{(r-1)}}(b^{(r-1)})
\in \{0, \dots, a^{(r-1)}-1\} \subset \mathbb{Z}_{\ge 0}.
\]

\medskip
\textit{Step 3: Back-substitution, integrality, and the closed formula.}

We prove by downward induction on $\ell = r, r-1, \dots, 1$ that
\[
x_\ell
=
\sum_{j=\ell-1}^{r-1}
b^{(j+1)}
\prod_{i=\ell-1}^{j-1}
\frac{a^{(i)}}{a^{(i+1)}}.
\tag{$F_\ell$}
\]

\textit{Base case} $\ell = r$: the sum reduces to the single term
$j = r-1$, giving
\[
x_r = b^{(r)} \cdot \prod_{i=r-1}^{r-2}(\cdots) = b^{(r)} \cdot 1 = b^{(r)},
\]
where the empty product (lower index exceeds upper index) equals $1$.
This agrees with Step~2.

\textit{Inductive step}: assume $(F_\ell)$ holds for some
$2 \le \ell \le r$.
Rearranging $(R_\ell)$,
\[
x_{\ell-1}
=
\frac{a^{(\ell-2)} x_\ell + b^{(\ell-1)}}{a^{(\ell-1)}}.
\]
Substituting $(F_\ell)$,
\begin{align*}
x_{\ell-1}
&=
\frac{a^{(\ell-2)}}{a^{(\ell-1)}}
\sum_{j=\ell-1}^{r-1}
b^{(j+1)}
\prod_{i=\ell-1}^{j-1}
\frac{a^{(i)}}{a^{(i+1)}}
+
\frac{b^{(\ell-1)}}{a^{(\ell-1)}}
\\
&=
\sum_{j=\ell-1}^{r-1}
b^{(j+1)}
\prod_{i=\ell-2}^{j-1}
\frac{a^{(i)}}{a^{(i+1)}}
+
b^{(\ell-1)}
\prod_{i=\ell-2}^{\ell-3}
\frac{a^{(i)}}{a^{(i+1)}}
\\
&=
\sum_{j=\ell-2}^{r-1}
b^{(j+1)}
\prod_{i=\ell-2}^{j-1}
\frac{a^{(i)}}{a^{(i+1)}},
\end{align*}
which is $(F_{\ell-1})$. This completes the induction.

Setting $\ell = 1$ in $(F_1)$ and noting that $a^{(0)} = k$, we obtain
\[
x_1
=
\sum_{j=0}^{r-1}
b^{(j+1)}
\prod_{i=0}^{j-1}
\frac{a^{(i)}}{a^{(i+1)}}.
\]
The product telescopes:
\[
\prod_{i=0}^{j-1}
\frac{a^{(i)}}{a^{(i+1)}}
=
\frac{a^{(0)}}{a^{(j)}}
=
\frac{k}{a^{(j)}},
\qquad j \ge 1,
\]
and for $j = 0$ the empty product equals $1 = k/a^{(0)} = k/k$,
so the formula $k/a^{(j)}$ holds uniformly for all $0 \le j \le r-1$.
Therefore
\[
x_1
=
k \sum_{j=0}^{r-1}
\frac{b^{(j+1)}}{a^{(j)} a^{(j+1)}}
\;\in\; \mathbb{Z},
\]
and by construction $\delta_k(a_0 x_1) = b_0$.

\medskip
\textit{Step 4: Minimality.}

By Lemma~\ref{lema:unicidad_ep} there is exactly one solution in
$\{0,\dots,k-1\}$, so
\[
x = \delta_k(x_1)
=
\delta_k\!\left(
k\sum_{j=0}^{r-1}
\frac{b^{(j+1)}}{a^{(j)}a^{(j+1)}}
\right)
\]
is that unique solution, which coincides with the minimal solution.
\end{proof}

\begin{example}\label{ex:minimal}
We find the unique solution $x \in \{0,\dots,12\}$ of
\[
\delta_{13}(5x) = 4.
\]
Here $k = 13$, $a_0 = 5$, $b_0 = 4$, and $\gcd(5,13) = 1$.

\medskip
\textit{Sequences $\{a^{(n)}\}$ and $\{b^{(n)}\}$.}
\[
\begin{array}{c|ccccc}
n       & 0  & 1 & 2               & 3             & 4             \\ \hline
a^{(n)} & 13 & 5 & \delta_5(13)= 3 & \delta_3(5)=2 & \delta_2(3)=1
\end{array}
\]
Thus $r = 4$. The sequence $\{b^{(n)}\}$ is computed as
$b^{(n)} = a^{(n-1)} - \delta_{a^{(n-1)}}(b^{(n-1)})$:
\[
\begin{array}{c|cccc}
n       & 1 & 2                       & 3                     & 4                     \\ \hline
b^{(n)} & 4 & 5 - \delta_5(4) = 1     & 3 - \delta_3(1) = 2   & 2 - \delta_2(2) = 2
\end{array}
\]

\medskip
\textit{Closed formula.}
The terms of the sum are
\[
\frac{b^{(1)}}{a^{(0)}a^{(1)}} = \frac{4}{65},
\quad
\frac{b^{(2)}}{a^{(1)}a^{(2)}} = \frac{1}{15},
\quad
\frac{b^{(3)}}{a^{(2)}a^{(3)}} = \frac{2}{6} = \frac{1}{3},
\quad
\frac{b^{(4)}}{a^{(3)}a^{(4)}} = \frac{2}{2} = 1.
\]
Using common denominator $195 = 3 \cdot 5 \cdot 13$:
\[
13 \cdot \frac{12 + 13 + 65 + 195}{195}
= 13 \cdot \frac{285}{195}
= 13 \cdot \frac{19}{13}
= 19.
\]
Hence
\[
x = \delta_{13}(19) = 19 - 13 = 6.
\]

\medskip
\textit{Verification.}\quad
$\delta_{13}(5 \cdot 6) = \delta_{13}(30) = 30 - 2\cdot 13 = 4.$\quad\checkmark
\end{example}

\section{Main Results}
\label{sec:main}

The case of a general affine equation $\delta_k(a_0 x + b) = c$
reduces to $\delta_k(a_0 x) = \delta_k(c-b)$ via additive
compatibility~\eqref{eq:addcomp}, and is applied below to
solve $\delta_p(6x+1)=0$ and $\delta_p(6x+5)=0$.

We derive explicit closed formulas for $g(2n,p)$. The coprimality
condition $\gcd(h,6p)=1$ requires $h$ to avoid the prime divisors
$2$, $3$, and $p$.

Integers coprime with $6$ are precisely those satisfying
$\delta_6(h) \in \{1,5\}$,
so every admissible integer $h$ can be written in one of the forms
\[
h = 6x+1 \qquad\text{or}\qquad h = 6x+5.
\]
The condition $p \mid h$ corresponds exactly to
\[
\delta_p(6x+1)=0 \qquad\text{and}\qquad \delta_p(6x+5)=0.
\]
The minimal solutions of these equations determine the residue
positions that must be excluded when counting integers coprime with
$6p$.

\begin{lemma}\label{lem:delta6_prime}
For every prime $p \ge 5$, $\,\delta_6(p) \in \{1,5\}$.
\end{lemma}

\begin{proof}
Since $p \ge 5$ is prime, $p$ is not divisible by $2$ or $3$,
hence $\gcd(p,6)=1$.
The elements of $\{0,1,2,3,4,5\}$ that are coprime with $6$ are
precisely $1$ and $5$, so $\delta_6(p) \in \{1,5\}$.
\end{proof}

\begin{lemma}\label{lema:simet}
Let $p \ge 2$, $\,a,b,c \in \{0,\dots,p-1\}$ with $\gcd(a,p)=1$
and $b+c=a$ in $\mathbb{Z}$.
Let $x_0$ and $x_0'$ denote the minimal solutions of
\[
\delta_p(ax+b)=0 \qquad\text{and}\qquad \delta_p(ax+c)=0,
\]
respectively. Then
\[
x_0 + x_0' + 1 = p.
\]
\end{lemma}

\begin{remark}
The hypothesis $b+c=a$ in $\mathbb{Z}$, together with
$a,b,c \in \{0,\dots,p-1\}$, implies $b+c < p$.
This is satisfied in all applications below, where
$(a,b,c)=(6,1,5)$ and $p \ge 5$.
\end{remark}

\begin{proof}
By hypothesis,
\[
\delta_p(ax_0+b)=0, \qquad \delta_p(ax_0'+c)=0.
\]
By additive compatibility~\eqref{eq:addcomp} and the hypothesis
$b+c=a$,
\[
\delta_p\bigl(a(x_0+x_0'+1)\bigr)
= \delta_p\bigl((ax_0+b)+(ax_0'+c)\bigr)
= \delta_p\bigl(\delta_p(ax_0+b)+\delta_p(ax_0'+c)\bigr)
= \delta_p(0+0)
= 0.
\]
Since $\gcd(a,p)=1$, Lemma~\ref{lema_pote} gives
$\delta_p(x_0+x_0'+1)=0$.
Because $x_0,x_0' \in \{0,\dots,p-1\}$,
\[
x_0+x_0'+1 \in [1,2p-1]_{\mathbb{Z}},
\]
and the only multiple of $p$ in this interval is $p$ itself,
so $x_0+x_0'+1=p$.
\end{proof}

\begin{lemma}\label{thm:sol_min6_funcional_1}
Let $p \ge 5$ be prime. The minimal solution of
\[
\delta_p(6x+1)=0
\]
is
\begin{equation}\label{eq_sola(p)}
a(p)
=
\frac{p-1}{6}\,M(p)
+
\frac{5p-1}{6}\,(1-M(p)),
\qquad
M(p):=\frac{5-\delta_6(p)}{4}.
\end{equation}
By Lemma~\ref{lem:delta6_prime}, $M(p)$ takes only integer values:
$M(p)=1$ if $\delta_6(p)=1$, and $M(p)=0$ if $\delta_6(p)=5$.
Equivalently,
\[
a(p)=\frac{p-1}{6}
\quad\text{if } \, \delta_6(p)=1,
\qquad
a(p)=\frac{5p-1}{6}
\quad\text{if } \, \delta_6(p)=5.
\]
\end{lemma}

\begin{proof}
Since $\gcd(6,p)=1$, by Lemma~\ref{lema:unicidad_ep} the equation
admits a unique solution in $\{0,\dots,p-1\}$.

The condition $\delta_p(6x+1)=0$ is equivalent to $6x+1=px_2$
for some integer $x_2$, which gives $6x = px_2-1$.
Applying $\delta_6$ to both sides, using additive and multiplicative
compatibility~\eqref{eq:addcomp},~\eqref{eq:multcomp}, and
Lemma~\ref{lema:cp2} to obtain $\delta_6(-1)=6-1=5$,
\[
0 = \delta_6(6x) = \delta_6(px_2-1)
= \delta_6\!\left(\delta_6(p)\,x_2 + 5\right).
\]

\textit{Case $\delta_6(p)=1$}:
\[
\delta_6(x_2+5)=0,
\]
whose smallest nonnegative solution is $x_2=1$, giving
\[
x = \frac{p-1}{6}.
\]

\textit{Case $\delta_6(p)=5$}:
\[
\delta_6(5x_2+5)=\delta_6(5(x_2+1))=0.
\]
Since $\gcd(5,6)=1$, this reduces to $\delta_6(x_2+1)=0$,
whose smallest nonnegative solution is $x_2=5$, giving
\[
x = \frac{5p-1}{6}.
\]

In both cases we verify $x \in \{0,\dots,p-1\}$:
for Case~1, $0 < (p-1)/6 < p$;
for Case~2, $(5p-1)/6 < p$ since $5p-1 < 6p$.
By Lemma~\ref{lema:unicidad_ep}, each is the unique, hence minimal,
solution.

The selector $M(p)$ encodes both cases, yielding formula
\eqref{eq_sola(p)}.
\end{proof}

\begin{lemma}\label{thm:sol_min6_funcional_5}
Let $p \ge 5$ be prime. The minimal solution of
\[
\delta_p(6x+5)=0
\]
is
\[
b(p) = p-1-a(p).
\]
\end{lemma}

\begin{proof}
Apply Lemma~\ref{lema:simet} with $a=6$, $b=1$, $c=5$.
The hypothesis $b+c=a$ holds since $1+5=6$ in $\mathbb{Z}$,
and $\gcd(6,p)=1$ since $p\ge 5$ is prime
(Lemma~\ref{lem:delta6_prime}).
Lemma~\ref{lema:simet} gives $x_0+x_0'+1=p$,
i.e.\ $b(p) = p-1-a(p)$.
\end{proof}

\subsection{Counting pairs with sum constraints}
\label{subsec:counting}

Throughout this subsection, $p \ge 2$ is fixed and
$s,t,r \in \{0,\dots,p-1\}$.
We define the indicator functions
\begin{equation}\label{eq:H_D_def}
H(x)
=
\begin{cases}
1 & x \ge 0,\\
0 & x < 0,
\end{cases}
\qquad
\mathbf{1}(x)
=
\begin{cases}
1 & x=0,\\
0 & x\ne 0.
\end{cases}
\end{equation}

\begin{lemma}\label{lem:conteo_excluyendo_a_general}
Let $a,r \in \{0,\dots,p-1\}$.
The number of pairs
$(s,t) \in (\{0,\dots,p-1\}\setminus\{a\})^2$
with $s+t=r$ and $s \le t$ is
\[
\kappa(r,a)
=
\left\lfloor\frac{r}{2}\right\rfloor+1-H(r-a).
\]
\end{lemma}

\begin{proof}
Every solution has the form $(s, r-s)$ with $0 \le s \le r-s$,
i.e.\ $0 \le s \le \lfloor r/2 \rfloor$, giving
$\lfloor r/2 \rfloor + 1$ pairs in total.

We now count how many of these contain $a$ in at least one component.
If $a > r$, then neither $s = a$ nor $t = r-s = a$ is possible for
any admissible $s \in \{0,\dots,\lfloor r/2\rfloor\}$, so no pair is
excluded; this is encoded by $H(r-a) = 0$.

If $a \le r$, then $s = a$ gives the pair $(a, r-a)$, which is
admissible since $a \le r-a \iff a \le r/2$ requires checking:
if $a \le \lfloor r/2 \rfloor$ then $(a,r-a)$ satisfies $s \le t$
and is the unique excluded pair;
if $a > \lfloor r/2 \rfloor$ (and $a \le r$) then $s = r-a$ gives
the pair $(r-a, a)$, again the unique excluded pair.
In both subcases exactly one pair is excluded; this is encoded by
$H(r-a) = 1$.

Therefore $\kappa(r,a) = \lfloor r/2 \rfloor + 1 - H(r-a)$.
\end{proof}

We define the diagonal indicator
\[
\tau(r,a)
=
(1-\delta_2(r))\,\mathbf{1}(2a-r),
\]
which equals $1$ precisely when $r$ is even and $r = 2a$,
i.e.\ when the diagonal pair $(a,a)$ would be present.
Define the vector
\[
\bm{\nu}(r,a)
=
\bigl(\kappa(r,a)-\tau(r,a),\;\tau(r,a)\bigr).
\]
Its first component counts strictly ordered pairs $(s,t)$ with
$s < t$, and its second flags the presence of the diagonal pair.

Since $r \in \{0,\dots,p-1\}$, we have $p+r \in \{p,\dots,2p-1\}$.
Hence the constraints $t=p+r-s \in \{0,\dots,p-1\}$ and $s \in \{0,\dots,p-1\}$
force $s$ to range over a contiguous interval.

\begin{lemma}\label{lem:conteo_p_r}
Let $a,r \in \{0,\dots,p-1\}$.
The number of pairs
$(s,t) \in (\{0,\dots,p-1\}\setminus\{a\})^2$
with $s+t=p+r$ and $s \le t$ is
\[
\bar{\kappa}(r,a,p)
=
\left\lfloor\frac{p-r}{2}\right\rfloor - H(a-r-1).
\]
\end{lemma}

\begin{proof}
From $t = p+r-s$ and $t \in \{0,\dots,p-1\}$, we need
$p+r-s \le p-1$, i.e.\ $s \ge r+1$.
Combined with $s \in \{0,\dots,p-1\}$, the admissible values of $s$
are $s \in \{r+1,\dots,p-1\}$.
The condition $s \le t = p+r-s$ gives $s \le (p+r)/2$, so
\[
s \in \left\{r+1,\dots,\left\lfloor\frac{p+r}{2}\right\rfloor\right\},
\]
which contains
\[
\left\lfloor\frac{p+r}{2}\right\rfloor - r
= \left\lfloor\frac{p-r}{2}\right\rfloor
\]
integers.

We now exclude pairs containing $a$.
If $a \le r$, then $a < r+1$, so $a$ does not appear as $s$ in
any admissible pair, and $t = p+r-a > p-1$ so $a$ cannot appear
as $t$ either; thus $H(a-r-1) = 0$ and no pair is excluded.

If $a \ge r+1$, then $s = a$ gives the admissible pair
$(a, p+r-a)$ with $t = p+r-a \in \{0,\dots,p-1\}$,
which must be excluded; this is encoded by $H(a-r-1) = 1$.

Therefore $\bar{\kappa}(r,a,p) =
\lfloor(p-r)/2\rfloor - H(a-r-1)$.
\end{proof}

Analogously to $\bm{\nu}$, define the vector
\[
\bar{\bm{\nu}}(r,a,p)
=
\bigl(\bar{\kappa}(r,a,p) - \tau(p-r,\,a-r),\;
      \tau(p-r,\,a-r)\bigr),
\]
which decomposes the count into strictly ordered and diagonal
contributions.

\begin{lemma}\label{lem:lambda}
Let $a,b,r \in \{0,\dots,p-1\}$.
The number of pairs
\[
(s,t)
\in
(\{0,\dots,p-1\}\setminus\{a\})
\times
(\{0,\dots,p-1\}\setminus\{b\})
\]
with $s+t=r$ is
\[
\lambda(r,a,b)
=
r+1-H(r-a)-H(r-b)+\mathbf{1}(a+b-r).
\]
\end{lemma}

\begin{proof}
Every solution of $s+t=r$ with $s,t \in \{0,\dots,p-1\}$ has the
form $(s,r-s)$ with $s \in \{0,\dots,r\}$, giving $r+1$ pairs in
total.

Define
\[
A = \{(s,t): s+t=r,\; s=a\},
\quad
B = \{(s,t): s+t=r,\; t=b\}.
\]
Then the desired count is
$(r+1) - |A \cup B| = (r+1) - |A| - |B| + |A \cap B|$.

\textit{Counting $|A|$.}
The pair $(a, r-a)$ belongs to $A$ if and only if
$r-a \in \{0,\dots,p-1\}$, i.e.\ $a \le r$.
Hence $|A| = H(r-a)$.

\textit{Counting $|B|$.}
The pair $(r-b, b)$ belongs to $B$ if and only if
$r-b \in \{0,\dots,p-1\}$, i.e.\ $b \le r$.
Hence $|B| = H(r-b)$.

\textit{Counting $|A \cap B|$.}
A pair belongs to both $A$ and $B$ if and only if $s=a$ and $t=b$,
i.e.\ $a+b=r$. Hence
$|A \cap B| = \mathbf{1}(a+b-r)$.

Substituting gives
\[
\lambda(r,a,b)
= r+1-H(r-a)-H(r-b)+\mathbf{1}(a+b-r). \qedhere
\]
\end{proof}

\begin{lemma}\label{lem:lambda_p}
Let $a,b,r \in \{0,\dots,p-1\}$.
The number of pairs
\[
(s,t)
\in
(\{0,\dots,p-1\}\setminus\{a\})
\times
(\{0,\dots,p-1\}\setminus\{b\})
\]
with $s+t=p+r$ is
\[
\bar{\lambda}(r,a,b,p)
=
p-r-1-H(a-r-1)-H(b-r-1)+\mathbf{1}(a+b-p-r).
\]
\end{lemma}

\begin{proof}
From $t = p+r-s$ and $t \in \{0,\dots,p-1\}$, we need $s \ge r+1$,
so $s \in \{r+1,\dots,p-1\}$, giving $p-r-1$ pairs in total.

Define
\[
A = \{(s,t): s+t=p+r,\; s=a\},
\quad
B = \{(s,t): s+t=p+r,\; t=b\}.
\]
The desired count is $(p-r-1) - |A \cup B|
= (p-r-1) - |A| - |B| + |A \cap B|$.

\textit{Counting $|A|$.}
The pair $(a, p+r-a)$ is admissible when $a \ge r+1$
(so that $s=a$ lies in $\{r+1,\dots,p-1\}$) and
$p+r-a \in \{0,\dots,p-1\}$ (which holds since $a \ge r+1 > 0$).
Hence $|A| = H(a-r-1)$.

\textit{Counting $|B|$.}
Similarly, $t=b$ gives $s = p+r-b$, admissible when $b \ge r+1$.
Hence $|B| = H(b-r-1)$.

\textit{Counting $|A \cap B|$.}
Both conditions hold simultaneously iff $s=a$ and $t=b$,
i.e.\ $a+b = p+r$.
Hence $|A \cap B| = \mathbf{1}(a+b-p-r)$.

Substituting gives
\[
\bar{\lambda}(r,a,b,p)
= p-r-1-H(a-r-1)-H(b-r-1)+\mathbf{1}(a+b-p-r). \qedhere
\]
\end{proof}

\subsection{Counting coprime pairs}
\label{subsec:coprime_pairs}

Every positive integer coprime with $6$ belongs to exactly one of
\begin{equation}\label{eq:conjuntoUV}
U := \{6k+1 : k \in \mathbb{Z}_{\ge 0}\},
\qquad
V := \{6k+5 : k \in \mathbb{Z}_{\ge 0}\}.
\end{equation}
Since $U \equiv 1 \pmod{6}$ and $V \equiv 5 \pmod{6}$, we have:
if $u_1,u_2\in U$, then $u_1+u_2\equiv 2\pmod{6}$; if $v_1,v_2\in V$,
then $v_1+v_2\equiv 4\pmod{6}$; and if $u\in U$, $v\in V$, then
$u+v\equiv 0\pmod{6}$.
Therefore, when $2n$ is written as a sum of two integers coprime with $6$,
the residue $\delta_3(n)$ determines the type of decomposition:
\begin{itemize}
\item if $\delta_3(n)=0$, one summand lies in $U$ and the other in $V$,
\item if $\delta_3(n)=1$, both summands lie in $U$,
\item if $\delta_3(n)=2$, both summands lie in $V$.
\end{itemize}

Define the selector function $h:\{0,1,2\}\to\mathbb{Z}$ by
\[
h(0)=3,\qquad h(1)=1,\qquad h(2)=5.
\]
This function encodes the constant term arising in the three possible
coprime decompositions. It admits the polynomial representation
\[
h(x)=3x^2-5x+3,
\]
which allows a unified expression for the auxiliary quantity
\begin{equation}\label{eq:omegas}
m(n) := \frac{n-h(\delta_3(n))}{3},
\qquad
\omega(n,p) := \frac{m(n)-\delta_p(m(n))}{p}
= \left\lfloor\frac{m(n)}{p}\right\rfloor.
\end{equation}

\begin{lemma}\label{lem:m_integer}
For all integers $n$, the quantity $m(n)$ is an integer.
\end{lemma}

\begin{proof}
We verify that $h(j) \equiv j \pmod{3}$ for $j \in \{0,1,2\}$:
$h(0)=3\equiv 0$, $h(1)=1\equiv 1$, $h(2)=5\equiv 2 \pmod{3}$.
Since $n \equiv \delta_3(n) \pmod{3}$, it follows that
$3 \mid (n-h(\delta_3(n)))$, so $m(n) \in \mathbb{Z}$.
\end{proof}

We define the counting vectors
\begin{align}
\bm{\eta}(n,p)
&:= \bigl(\omega(n,p)+1,\;
          \lfloor\omega(n,p)/2\rfloor+1\bigr),
\label{eq:eta}\\
\overline{\bm{\eta}}(n,p)
&:= \bigl(\omega(n,p),\;
          \lfloor(\omega(n,p)-1)/2\rfloor+1\bigr),
\label{eq:etabar}
\end{align}
and recall the residue vectors $\bm{\nu}(r,a)$ and
$\overline{\bm{\nu}}(r,a,p)$ defined in
Section~\ref{subsec:counting}.
We define the scalar product
$(x_1,x_2)\cdot(y_1,y_2) := x_1y_1+x_2y_2$.

\begin{lemma}\label{lema:conteo_soluciones}
Let $2n$ be a positive even integer with $\delta_3(n)\in\{1,2\}$,
and let $p \ge 5$ be prime. Define
\[
f(n,p) :=
\begin{cases}
a(p) & \text{if } \delta_3(n)=1,\\
b(p) & \text{if } \delta_3(n)=2.
\end{cases}
\]
The number of pairs $(h,k) \in \mathbb{Z}_{>0}^2$ with
$h+k=2n$, $h \le k$, $\gcd(h,6p)=\gcd(k,6p)=1$ is
\begin{equation}\label{eq:Q}
Q(n,p)
=
\bm{\eta}(m(n),p)
\cdot
\bm{\nu}\!\bigl(\delta_p(m(n)),\,f(n,p)\bigr)
+
\overline{\bm{\eta}}(m(n),p)
\cdot
\overline{\bm{\nu}}\!\bigl(\delta_p(m(n)),\,f(n,p),\,p\bigr).
\end{equation}
\end{lemma}

\begin{proof}
We treat $\delta_3(n)=1$; the case $\delta_3(n)=2$ is analogous
with $a(p)$ replaced by $b(p)$.

\textit{Step 1: Reduction to a sum of residues.}
Every admissible decomposition with $\gcd(h,6)=\gcd(k,6)=1$
and $\delta_3(n)=1$ has the form
$(6z_1+1)+(6z_2+1)=2n$ with $z_1,z_2 \in \mathbb{Z}_{\ge 0}$.
This gives $z_1+z_2=(n-1)/3=m(n)$,
where the last equality uses $h(1)=1$ and Lemma~\ref{lem:m_integer}.

\textit{Step 2: Exclusion of multiples of $p$.}
By Lemma~\ref{thm:sol_min6_funcional_1},
$p \mid (6z_i+1) \iff \delta_p(z_i)=a(p)$.
Hence admissible values satisfy
$\delta_p(z_i) \in \{0,\dots,p-1\}\setminus\{a(p)\}$.
Write $z_i=px_i+s_i$ with $x_i \in \mathbb{Z}_{\ge 0}$ and
$s_i \in \{0,\dots,p-1\}\setminus\{a(p)\}$.
Substituting into $z_1+z_2=m(n)$ gives
\[
p(x_1+x_2)+s_1+s_2=m(n).
\]

\textit{Step 3: Residue constraints.}
Set $r:=\delta_p(m(n))$. Since $s_1,s_2 \in \{0,\dots,p-1\}$,
we have $s_1+s_2 \in \{0,\dots,2p-2\}$.
The only values in this range congruent to $r$ modulo $p$ are
$r$ and $p+r$, so exactly one of
\[
s_1+s_2=r
\qquad\text{or}\qquad
s_1+s_2=p+r
\]
holds for each admissible pair.

\textit{Step 4: Counting lifts.}
If $s_1+s_2=r$, then $x_1+x_2=\omega(n,p)$,
giving $\omega(n,p)+1$ pairs $(x_1,x_2) \in \mathbb{Z}_{\ge 0}^2$.
If $s_1+s_2=p+r$, then $x_1+x_2=\omega(n,p)-1$,
giving $\omega(n,p)$ pairs.

\textit{Step 5: Ordering condition.}
Since $6z_1+1 \le 6z_2+1 \iff z_1 \le z_2$, the condition
$h \le k$ translates directly to $z_1 \le z_2$.
For a symmetric residue pair $s_1=s_2$, the ordering requires
$x_1 \le x_2$, reducing the lift count to
$\lfloor\omega(n,p)/2\rfloor+1$ or
$\lfloor(\omega(n,p)-1)/2\rfloor+1$ respectively.
The vectors $\bm{\nu}$ and $\overline{\bm{\nu}}$ encode the
admissible residue pairs split into strictly ordered and diagonal
contributions; $\bm{\eta}$ and $\overline{\bm{\eta}}$ encode
the corresponding lift counts.
Combining via the scalar product gives~\eqref{eq:Q}.
\end{proof}

\begin{lemma}\label{lema:conteo_soluciones_0}
Let $2n$ be a positive even integer with $\delta_3(n)=0$,
and let $p \ge 5$ be prime.
The number of pairs $(h,k) \in \mathbb{Z}_{>0}^2$ with
$h+k=2n$, $h \le k$, $\gcd(h,6p)=\gcd(k,6p)=1$ is
\begin{equation}\label{eq:Q0}
Q_0(n,p)
=
\bigl(\omega(n,p)+1\bigr)\,
\lambda\!\bigl(\delta_p(m(n)),\,a(p),\,b(p)\bigr)
+
\omega(n,p)\,
\bar{\lambda}\!\bigl(\delta_p(m(n)),\,a(p),\,b(p),\,p\bigr).
\end{equation}
\end{lemma}

\begin{proof}
Since $\delta_3(n)=0$, every admissible decomposition has the form
$(6z_1+1)+(6z_2+5)=2n$ with $z_1,z_2 \in \mathbb{Z}_{\ge 0}$,
giving $z_1+z_2=(n-3)/3=m(n)$ using $h(0)=3$.

By Lemmas~\ref{thm:sol_min6_funcional_1}
and~\ref{thm:sol_min6_funcional_5},
$p \nmid (6z_1+1) \iff \delta_p(z_1) \ne a(p)$ and
$p \nmid (6z_2+5) \iff \delta_p(z_2) \ne b(p)$.
Write $z_i=px_i+s_i$ with
$s_1 \in \{0,\dots,p-1\}\setminus\{a(p)\}$ and
$s_2 \in \{0,\dots,p-1\}\setminus\{b(p)\}$.
Substituting gives $p(x_1+x_2)+s_1+s_2=m(n)$.

Setting $r:=\delta_p(m(n))$, the pair $(s_1,s_2)$ satisfies
$s_1+s_2=r$ or $s_1+s_2=p+r$ by the same argument as above.

For $s_1+s_2=r$: Lemma~\ref{lem:lambda} gives
$\lambda(r,a(p),b(p))$ admissible residue pairs, each lifting to
$\omega(n,p)+1$ pairs $(x_1,x_2) \in \mathbb{Z}_{\ge 0}^2$.

For $s_1+s_2=p+r$: Lemma~\ref{lem:lambda_p} gives
$\bar{\lambda}(r,a(p),b(p),p)$ admissible residue pairs, each
lifting to $\omega(n,p)$ pairs.

Since $h=6z_1+1 < 6z_2+5=k$ whenever $z_1=z_2$
(as $1 < 5$), and $h \le k \iff z_1 \le z_2$ otherwise,
no diagonal correction is needed.
Multiplying and summing gives~\eqref{eq:Q0}.
\end{proof}

\begin{theorem}\label{thm:main}
Let $p \ge 5$ be prime.
For every positive even integer $2n$,
\[
g(2n,p)
=
\begin{cases}
Q(n,p)   & \text{if } \delta_3(n)\in\{1,2\},\\
Q_0(n,p) & \text{if } \delta_3(n)=0,
\end{cases}
\]
where $Q(n,p)$ and $Q_0(n,p)$ are given by
\eqref{eq:Q} and~\eqref{eq:Q0}.
\end{theorem}

\begin{proof}
The cases $\delta_3(n)\in\{0,1,2\}$ are exhaustive and mutually
exclusive. For $\delta_3(n)\in\{1,2\}$, the result is
Lemma~\ref{lema:conteo_soluciones}. For $\delta_3(n)=0$, it is
Lemma~\ref{lema:conteo_soluciones_0}. Both lemmas apply to all
positive even integers $2n$ and primes $p \ge 5$.
\end{proof}

\section{Computational Validation}
\label{sec:validation}

To validate Theorem~\ref{thm:main}, we implemented two independent
procedures and compared them pointwise over all even integers
$2n \le 10^5$ and primes $p \in \{5,7,11,13,17,19,23\}$,
giving a total of $350{,}000$ test cases.
Perfect agreement was observed in all cases.
The implementation uses equivalent arithmetic expressions that
produce identical results to the formulas in
Theorem~\ref{thm:main}.

\subsection{Closed-form algorithm}

For fixed $p$, the parameters $a(p)$ and $b(p)$ are precomputed
once in $O(\log p)$ operations via Theorem~\ref{thm:min_sol_afrc}.
Thereafter, given any even integer $2n \ge 2$, the formula in
Theorem~\ref{thm:main} evaluates $g(2n,p)$ using only integer
arithmetic and a constant number of elementary operations,
giving complexity $O(1)$ per evaluation.

\begin{algorithm}[H]
\caption{Closed-form count of $g(2n,p)$}
\begin{algorithmic}[1]
\Require $N \ge 2$ even, prime $p \ge 5$
\State $n \gets N/2$
\If{$\delta_3(n) = 0$}
    \State $m \gets (n-3)/3$
    \State \Return $Q_0(n,p)$
\Else
    \State $m \gets (n - h(\delta_3(n)))/3$
    \State \Return $Q(n,p)$
\EndIf
\end{algorithmic}
\end{algorithm}

\subsection{Brute-force oracle}

The oracle iterates over $1 \le h \le n$, sets $k = N-h$,
and checks $\gcd(h,6p)=\gcd(k,6p)=1$ directly.
Time complexity: $O(n)$.

\begin{algorithm}[H]
\caption{Brute-force oracle}
\begin{algorithmic}[1]
\Require $N \ge 2$ even, prime $p \ge 5$
\State $n \gets N/2$
\State $\mathrm{total} \gets 0$
\For{$h = 1$ \textbf{to} $n$}
    \State $k \gets N - h$
    \If{$\gcd(h,6p)=1$ \textbf{and} $\gcd(k,6p)=1$}
        \State $\mathrm{total} \gets \mathrm{total} + 1$
    \EndIf
\EndFor
\State \Return $\mathrm{total}$
\end{algorithmic}
\end{algorithm}

\subsection{Coprime decomposition diagrams}

Figure~\ref{fig:diagrama} shows the graphs $(2n,g(2n,p))$ for
$p \in \{5,7,11\}$ and $2n \le 10^5$.
The piecewise affine structure governed by residue classes
modulo $3$ and $p$ is clearly visible; the parallel affine
families correspond to the residue classes of $n$ modulo $3p$,
consistent with Theorem~\ref{thm:main}.

\begin{figure}[h]
\centering
\includegraphics[width=0.85\linewidth]{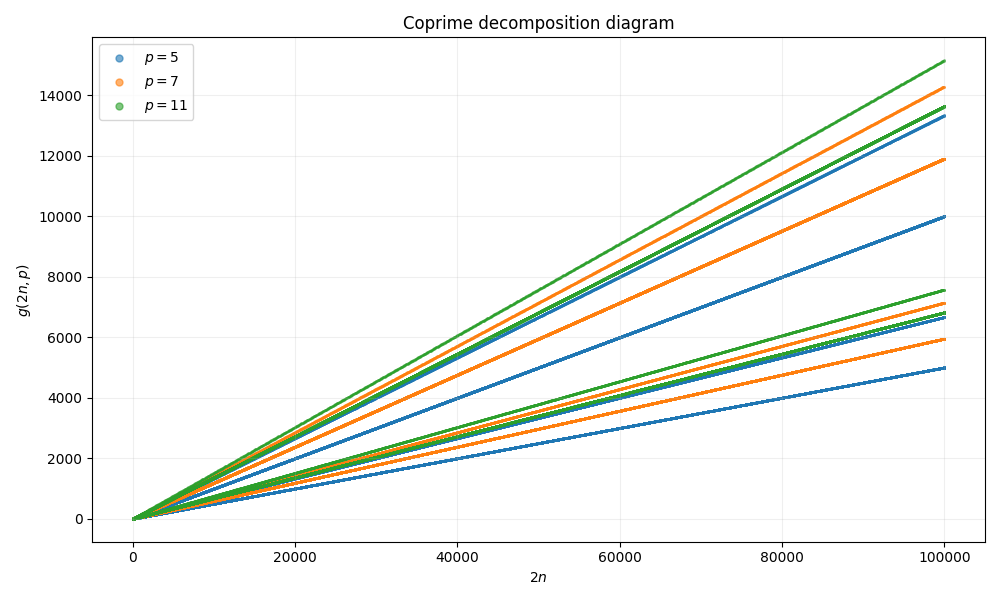}
\caption{Graphs of $g(2n,p)$ for $p\in\{5,7,11\}$ and
         $2n\le 10^5$.}
\label{fig:diagrama}
\end{figure}

\paragraph{Code availability.}
Source code for both procedures is publicly available at
\url{https://github.com/radix-place/coprime_count}.

\section{Conclusions}
\label{sec:conclusions}

We derived explicit closed-form formulas for the counting function
\[
g(2n,p)
=
\#\bigl\{(h,k)\in\mathbb{Z}_{>0}^2 :
h+k=2n,\; h \le k,\;
\gcd(h,6p)=\gcd(k,6p)=1\bigr\},
\]
expressed in terms of elementary remainder operations and step
functions. For fixed $p$, after precomputing the parameters
$a(p)$ and $b(p)$ in $O(\log p)$ operations via the Euclidean
algorithm, each evaluation of $g(2n,p)$ requires $O(1)$
operations, in contrast to the $O(n)$ brute-force approach.
The formulas were validated against a brute-force oracle for
all $2n \le 10^5$ and $p \in \{5,7,11,13,17,19,23\}$,
with perfect agreement in all cases.

The formulas make explicit a piecewise affine structure of
$g(2n,p)$ along arithmetic progressions of $n$, governed by
residue classes modulo $3$ and $p$.
This structure is directly encoded in the minimal solutions
$a(p)$ and $b(p)$ of the congruences
$6x \equiv -1 \pmod{p}$ and $6x \equiv -5 \pmod{p}$,
which act as the key arithmetic parameters of the problem.

The approach extends naturally to exclusion sets of the form
$\{2,3,p_1,\dots,p_k\}$ for any finite collection of primes,
at the cost of additional case analysis governed by residues
modulo $6p_1\cdots p_k$.
We leave this generalization for future work.

% --- Bibliography ---
\bibliographystyle{plainnat}
\bibliography{references}

@book{HardyWright2008,
  author    = {G. H. Hardy and E. M. Wright},
  title     = {An Introduction to the Theory of Numbers},
  edition   = {6},
  publisher = {Oxford University Press},
  year      = {2008}
}

@book{BrentZimmermann2010,
  author    = {R. P. Brent and P. Zimmermann},
  title     = {Modern Computer Arithmetic},
  publisher = {Cambridge University Press},
  year      = {2010}
}

@book{Nathanson2000,
  author    = {Melvyn B. Nathanson},
  title     = {Additive Number Theory: The Classical Bases},
  publisher = {Springer},
  year      = {2000}
}

@book{halberstam-richert,
  author = {Halberstam, H. and Richert, H.-E.},
  title = {Sieve Methods},
  publisher = {Academic Press},
  year = {1974}
}

@book{iwaniec-kowalski,
  author = {Iwaniec, Henryk and Kowalski, Emmanuel},
  title = {Analytic Number Theory},
  publisher = {American Mathematical Society},
  year = {2004}
}

\end{document}